\documentclass[11pt,amstex,amssymb]{amsart}
\usepackage{amsmath,amsthm,amsfonts,amssymb,amscd}
\usepackage[latin1]{inputenc}
\usepackage[all]{xy}
\usepackage[dvips]{graphicx}
\usepackage{amsmath}

\usepackage{amsthm}
\usepackage{amsfonts}
\usepackage{xcolor}
\usepackage{amsmath}
\usepackage{amssymb}
\usepackage{amsthm}

\def\fa{{\mathcal{F}}}
\def\O{{\mathcal{O}}}

\def\M{{\mathcal{M}}}
\def\D{{\mathcal{D}}}
\def\U{{\mathcal{U}}}
\def\B{{\mathcal{B}}}

\title[Irreducible holonomy groups]{Irreducible holonomy groups and first
integrals for holomorphic foliations}
\author{V. Le\'on}
\author{M. Martelo}
\author{B. Sc\'ardua}
\address{V. Le\'on. ILACVN - CICN, Universidade Federal da
Integração Latino-Americano, Parque tecnológico de Itaipu,
Foz do Iguaçu-PR, 85867-970 - Brazil}
\email{victor.leon@unila.edu.br}

\address{M. Martelo. Instituto de Matem\'atica -
Universidade Federal Fluminense, Niter\'oi -
Rio de Janeiro-RJ, 24210-201 - Brazil}
\email{mitchaelmartelo@id.uff.br}

\address{B. Sc\'ardua. Instituto de Matem\'atica -
Universidade Federal do Rio de Janeiro,
CP. 68530-Rio de Janeiro-RJ, 21945-970 - Brazil}
\email{scardua@im.ufrj.br}


\date{}

\begin{document}

\begin{abstract}
We study  groups of germs of  complex diffeomorphisms  having a
property called {\it irreducibility}. The notion is motivated by the
similar property of the fundamental group of the complement of an
irreducible hypersurface in the complex projective space. Natural
examples of such groups of germ maps are given by holonomy groups
and monodromy groups of integrable systems (foliations) under
certain conditions. We prove some finiteness results for these
groups extending previous results in \cite{CL}. Applications are
given to the framework of germs of holomorphic foliations. We prove
the existence of first integrals under certain irreducibility or
more general  conditions on the tangent cone of the foliation after
a punctual blow-up.
 \end{abstract}

\maketitle
\newtheorem{theorem}{Theorem}
\renewcommand*{\thetheorem}{\Alph{theorem}}

\newtheorem{Theorem}{Theorem}[section]
\newtheorem{Corollary}{Corollary}[section]
\newtheorem{Proposition}{Proposition}[section]
\newtheorem{Lemma}{Lemma}[section]
\newtheorem{Claim}{Claim}[section]
\newtheorem{Definition}{Definition}[section]
\newtheorem{Example}{Example}[section]
\newtheorem{Remark}{Remark}[section]
\newtheorem*{cltheorem}{Theorem}
\newtheorem{Question}{Question}[section]

\newcommand\virt{\rm{virt}}
\newcommand\SO{\rm{SO}}
\newcommand\G{\varGamma}
\newcommand\Om{\Omega}
\newcommand\Kbar{{K\kern-1.7ex\raise1.15ex\hbox to 1.4ex{\hrulefill}}}
\newcommand\codim{\rm{codim}}
\renewcommand\:{\colon}
\newcommand\s{\sigma}
\def\vol#1{{|{\bfS}^{#1}|}}

\def\fa{{\mathcal F}}
\def\H{{\mathcal H}}
\def\O{{\mathcal O}}
\def\P{{\mathcal P}}
\def\L{{\mathcal L}}
\def\C{{\mathcal C}}
\def\Z{{\mathcal Z}}

\def\M{{\mathcal M}}
\def\N{{\mathcal N}}
\def\R{{\mathcal R}}
\def\ea{{\mathcal e}}
\def\Oa{{\mathcal O}}
\def\ee{{\bfE}}

\def\A{{\mathcal A}}
\def\B{{\mathcal B}}
\def\H{{\mathcal H}}
\def\V{{\mathcal V}}
\def\U{{\mathcal U}}
\def\al{{\alpha}}
\def\be{{\beta}}
\def\ga{{\gamma}}
\def\Ga{{\Gamma}}
\def\om{{\omega}}
\def\Om{{\Omega}}
\def\La{{\Lambda}}
\def\ov{\overline}
\def\dd{{\bfD}}
\def\pp{{\bfP}}

\def\nn{{\mathbb N}}
\def\zz{{\mathbb Z}}
\def\bq{{\mathbb Q}}
\def\bp{{\mathbb P}}
\def\bd{{\mathbb D}}
\def\bh{{\mathbb H}}
\def\te{{\theta}}
\def\rr{{\mathbb R}}
\def\bb{{\mathbb B}}
\def\pp{{\mathbb P}}
\def\dd{{\mathbb D}}
\def\zz{{\mathbb Z}}
\def\qq{{\mathbb Q}}
\def\hh{{\mathbb H}}
\def\nn{{\mathbb N}}
\def\LL{{\mathbb L}}
\def\co{{\mathbb C}}
\def\qq{{\mathbb Q}}
\def\na{{\mathbb N}}
\def\esima{${}^{\text{\b a}}$}
\def\esimo{${}^{\text{\b o}}$}
\def\lg{\lambdangle}
\def\rg{\rangle}
\def\ro{{\rho}}
\def\lV{\left\Vert}
\def\rV{\right\Vert }
\def\lv{\left\vert}
\def\rv{\right\vert }
\def\Sa{{\mathcal S}}
\def\D{{\mathcal D  }}

\def\si{{\bf S}}
\def\ve{\varepsilon}
\def\vr{\varphi}
\def\lV{\left\Vert }
\def\rV{\right\Vert}
\def\lv{\left\vert }
\def\rv{\right\vert}
\def\Range{\rm{{R}}}
\def\vol{\rm{{Vol}}}
\def\ind{\rm{{i}}}

\def\Int{\rm{{Int}}}
\def\Dom{\rm{{Dom}}}
\def\supp{\rm{{supp}}}
\def\Aff{\mbox{Aff}}
\def\Exp{\rm{{Exp}}}
\def\Hom{\rm{{Hom}}}
\def\codim{\rm{{codim}}}
\def\cotg{\rm{{cotg}}}
\def\dom{\rm{{dom}}}
\def\Sa{\mathcal{{S}}}

\def\VIP{\rm{{VIP}}}
\def\argmin{\rm{{argmin}}}
\def\Sol{\rm{{Sol}}}
\def\Ker{\rm{{Ker}}}
\def\Sat{\rm{{Sat}}}
\def\diag{\rm{{diag}}}
\def\rank{\rm{{rank}}}
\def\Sing{\rm{{Sing}}}
\def\sing{\rm{{sing}}}
\def\hot{\rm{{h.o.t.}}}

\def\Fol{\rm{{Fol}}}
\def\grad{\rm{{grad}}}
\def\id{\rm{{id}}}
\def\Id{\rm{{Id}}}
\def\sep{\rm{{Sep}}}
\def\Aut{\rm{{Aut}}}
\def\Sep{\rm{{Sep}}}
\def\Res{\rm{{Res}}}
\def\ord{\rm{{ord}}}
\def\h.o.t.{\rm{{h.o.t.}}}
\def\Hol{\mbox{Hol}}
\def\Diff{\mbox{Diff}}
\def\SL{\rm{{SL}}}

\tableofcontents
\section{Introduction}

Let $C\subset \mathbb P^2$ be an algebraic curve of degree $\nu$ in
the complex projective plane. By a classical theorem of Zariski and
Fulton-Deligne (\cite{Deligne,Fulton}), if $C\subset \mathbb P^2$ is
irreducible and smooth or with only normal double crossings
singularities, then the fundamental group of the complement $\mathbb
P^2 \setminus C$ is finite abelian, isomorphic to $\mathbb Z / \nu
\mathbb Z$. Indeed, in Deligne's paper it is observed that, in the
irreducible case, the fundamental group of the complement $\mathbb
P^2\setminus C$ is isomorphic to the local fundamental group
$\pi_1(V\setminus C)$ where $V$ is a suitable arbitrarily small
neighborhood of $C$ in $\mathbb P^2$. As a consequence, the group
$\pi_1(\mathbb P^2 \setminus C)$ has the following irreducibility
property: given two simple loops $\alpha, \beta $ in $\mathbb
P^2\setminus C$,  there is a loop $\delta \subset \mathbb P^2
\setminus C$ that conjugates the classes $[\alpha], [\beta] \in
\pi_1(\mathbb P^2\setminus C)$. For this it is not necessary to
assume that the curve is smooth or has only double points as
singularities, just the irreducibility of the curve is required. We
shall use this as a motivation for the definition of  an {\it
irreducibility} property for a  group of germs of analytic
diffeomorphisms.

Expand a  germ of a complex diffeomorphism $f$ at the origin $0 \in
\mathbb C$ as $f(z)= e^{2 \pi i \lambda} z + a_{k+1} z^{k+1} + ...$.
The {\it multiplier} $f^\prime(0)=e^{2 \pi i \lambda}$ does not
depend on the coordinate system. We shall say that the germ $f \in
\Diff(\mathbb C,0)$ is {\it flat} or {\it tangent to the identity}
if $f^\prime(0) =1$.

Using these ideas and some features from arithmetics,  in \cite{CL}
the authors  prove a finiteness theorem for groups of germs of
complex diffeomorphisms  in one variable. This reads as follows:

\begin{Theorem}[\cite{CL},  Theorem~1
page 221] \label{Theorem:CL} Let $G$ be a subgroup of $\Diff(\mathbb
C,0)$ generated by elements $f_1,...,f_{\nu+1}$. Assume that:
\begin{enumerate}
\item $f_k(z)= \mu z + \cdots$ with $\mu=e^{2 i\pi/ \nu+1}$ and
$f_{\nu+1} \circ \cdots \circ f_1=\id$.

\item the $f_k$ are pairwise conjugate in $G$.

\item $\nu+1 = p^s$ with $p $ prime and $s\in \mathbb N^*$.
\end{enumerate}
Then $G$ is a finite group; in particular $G$ is conjugate to a
finite group of rotations fixing $0$.

\end{Theorem}
Here we have $\mathbb N=\{0,1,2,...\}$ and $\mathbb N ^*= \mathbb N
\setminus \{0\}=\{1,2,...\}$. Notice that if $\nu+1=2$ so
that $G=<f_1,f_2>$ where $f_1\circ f_2=\id$. This implies that
$G=<f_1>$ is cyclic. Since $f_1$ and $(f_1)^{-1}$ are conjugate in
$G$ we have $f_1=(f_1)^{-1}$ and therefore $(f_1)^2=\id$. Thus $G$
is finite of order $\leq 2$.

 As a consequence of Theorem~\ref{Theorem:CL} the same authors obtain a theorem
(cf. \cite{CL} Theorem~0 page 218) about the existence of a
holomorphic first integral for a germ of holomorphic foliation of
codimension one singular at the origin $0 \in\mathbb C^n, n \geq 2$
provided that: (i) the foliation is not dicritic; (ii) the tangent
cone of the foliation is irreducible of degree $\nu +1=p^s$ where
$p$ is a prime number and $s\in \mathbb N^*$.

It is our aim in this paper to study further the  classification of
groups of germs of complex diffeomorphisms exhibiting this
irreducibility property and possible  applications of this to the
study of holomorphic foliations singularities. We shall then study
groups for which the  irreducibility holds but with a number of
generators not necessarily of the form $p^s$ for $p$ prime.

 For simplicity we consider the group $\Diff(\mathbb C,0)$ of germs of
complex diffeomorphisms fixing the origin $0 \in \co$. We shall
start with a definition:

\begin{Definition}[irreducible group]
\label{Definition:irreduciblegroup}
{\rm A subgroup $G \subset \Diff(\mathbb C,0)$ is {\it irreducible} if it admits a
finite set of generators
$f_1,f_2,\ldots,f_{\nu +1}\in G$ such that:
\begin{enumerate}
\item[(a)] $f_1\circ f_2\circ\ldots\circ f_{\nu+1}=\id$.
\item[(b)] $f_i$ and $f_j$ are conjugate in $G$ for all $i,j$.
    \end{enumerate}}

\end{Definition}
The above maps $f_j$ will be called {\it basic generators} of the
group $G$. By definition an irreducible group is finitely generated.
The above definition does not exclude the possibility that we have a
repetition of basic generators, ie.,  $f_i=f_j$. Note also that an
irreducible abelian group is finite cyclic: indeed, since the group
is abelian we have $f_i=f_j$, for all $i,j$. Therefore the group is
generated by an element $f_1$ of finite order because $f_1^{\nu +1}
=\id$.

Let us give some examples illustrating our above definition:

\begin{Example}
{\rm Every cyclic subgroup {\it of finite order} $G\subset
\Diff(\mathbb C,0)$ is irreducible. In particular every finite
subgroup $G \subset \Diff(\mathbb C,0)$ is irreducible. Indeed, such
a group is analytically conjugate to a cyclic group  of rational
rotations say, $G=<z\mapsto e^{2 \pi i /\nu} z>$ for some $\nu
\in\mathbb N^*$. On the other hand a cyclic group is irreducible
only if it has finite order. If $G\subset \Diff(\mathbb C,0)$ is
abelian and irreducible then $G$ is also finite. Indeed, all
generators $f_1,...,f_{\nu+1}$ in the definition of irreducible
group are conjugate so $f_1=\ldots=f_{\nu+1}$ and then
$f_1^{\nu+1}=\id$. This shows that $G$ is finite. Later on we shall
give examples of (nonabelian) irreducible groups (see
\S~\ref{section:examples}).}
\end{Example}

A (more geometrical) example is given below:

\begin{Example}

{\rm Let $C\subset \mathbb P^2$ be an algebraic curve of degree
$\nu$.  By a classical theorem of Zariski and Fulton-Deligne, if
$C\subset \mathbb P^2$  is irreducible and smooth or with only
normal double crossings singularities, then the fundamental group of
the complement $\mathbb P^2 \setminus C$ is finite abelian,
isomorphic to $\mathbb Z / \nu \mathbb Z$. If this the case, then
any representation $\vr\colon \pi_1(\mathbb P^2\setminus C) \to
\mathbb \Diff(\mathbb C,0)$ has as image an irreducible group. For
instance, assume that we have a closed rational one-form $\Omega$ on
$\mathbb P^2$ with polar divisor $(\Omega)_\infty =C$ irreducible.
Then the monodromy of $\Omega$ gives an irreducible subgroup of
$\Diff(\mathbb C,0)$. }
\end{Example}

Now we extend   Theorem~\ref{Theorem:CL} for the case some of
generators has a multiplier which is a  root of the unity of order
$p^s$ for a prime $p$ and natural number $s \geq 0$, {\sl but
without further restrictions on the number of generators}. This
includes the case where some of the basic generators is tangent to
the identity.

\begin{theorem}
\label{Theorem:2irreducibledim1} Let $G\subset \Diff(\mathbb C,0)$
be an irreducible  group.  If the multiplier of some basic generator
is a root of the unit of order $p^s$, where $p$ is prime and
$s\in\mathbb{N}$, then $G$ is a finite group.
\end{theorem}

A straightforward consequence of
Theorem~\ref{Theorem:2irreducibledim1} is that if some basic
generator is tangent to the identity ( ie., $s=0$) then $G$ is
trivial. Indeed, a finite subgroup $G\subset \Diff(\mathbb C,0)$ is
analytically linearizable.

 The fact that the basic generators are conjugate in the group $G$, not
only in $\Diff(\mathbb C,0)$ (condition (b) in
Definition~\ref{Definition:irreduciblegroup} above),  is essential
(cf. Example~\ref{Example:essential}).

Let us consider a germ of codimension one holomorphic foliation with
a singularity at the origin $0 \in \mathbb C^n, n \geq 2$. This is
given by an integrable germ of a holomorphic one-form $\omega$ at
$0$. The integrability condition writes $\omega \wedge d \omega=0$
and we may assume that $codim \sing(\omega)\geq 2$. In the case
where $codim \sing(\omega) \geq 3$ a theorem of Malgrange
(\cite{malgrange}) assures the existence of a holomorphic first
integral, ie., a germ of a holomorphic function $f \in \mathcal O_n$
such that $ df \wedge \omega=0$. Denote by $\pi \colon \tilde
{\mathbb C^n} \to \mathbb C^n$ the blow-up at the origin of $\mathbb
C^n$. The {\it tangent cone} of $\fa$ is defined as the intersection
$C(\fa)$ of the singular set of the lifted foliation $\tilde
\fa=\pi^*(\fa)$ with the exceptional divisor $E=\pi^{-1}(0)\cong
\mathbb P^{n-1}$, ie., $C(\fa) = \sing(\tilde \fa) \cap E$
(\cite{Cerveau-Mattei}). This is a (not necessarily irreducible)
algebraic subset of $\mathbb P^{-1}$ of codimension $\geq 1$. Assume
now that $\fa$ is non dicritical, ie., that $E$ is invariant by
$\tilde \fa$. In this case we have a leaf $\tilde L= E\setminus
C(\fa)$ of $\tilde \fa$. This leaf has a holonomy group called {\it
projective holonomy} of the foliation $\fa$. It can be viewed as
follows: take a point $\tilde q \in \tilde L$ and a transverse disc
$\Sigma$ to $\tilde L$ with center at $\tilde q$. Then consider the
representation of the fundamental group $\pi(\tilde L,\tilde q)$ in
the group $\Diff(\Sigma, \tilde q)$ of germs of complex
diffeomorphisms of $\Sigma$ fixing $\tilde q$, obtained as the image
of the holonomy maps of the loops $\gamma$ in $\tilde L$ based at
$\tilde q$. The projective holonomy group is the image $H(\fa)$ of
this representation, that can may be also viewed as a subgroup of
$\Diff(\mathbb C,0)$ under an identification $(\Sigma, \tilde q)=
(\mathbb C,0)$. Given a codimension $\geq 2$ component
$\Lambda_2\subset C(\fa)$ there is an isomorphism between
fundamental groups $\pi_1(E) \cong \pi_1(E \setminus \Lambda_2)$
because the real codimension of $\Lambda_2$ in $E$ is $\geq 4$. Thus
if $C_1(\fa)\subset  C(\fa)$ denotes the union of all codimension
one irreducible components of $C(\fa)$ we have $\pi_1(E\setminus
C_1(\fa))\cong \pi_1(E\setminus  C(\fa))=\pi_1(\tilde L)$. Given a
small simple loop $\gamma \subset E \setminus C(\fa)$ around an
irreducible component of codimension one of $C(\fa)$ (such loops
generate the fundamental group $\pi_1(\mathbb P^2 \setminus C(\fa))$
\cite{Deligne}) we shall call the corresponding holonomy map a {\it
basic holonomy generator} of the projective holonomy of $\fa$.

Assume now that the codimension one component of $C(\fa)$ is
irreducible. The projective holonomy group is the representation of
an group which by Deligne's theorem mentioned in the beginning, has
the irreducibility property. Thus $H(\fa)\subset \Diff(\mathbb C,0)$
is an irreducible group, having as basic generators  in
Definition~\ref{Definition:irreduciblegroup}, the above introduced
basic holonomy generators of the projective holonomy. Notice that we
do know whether $ C(\fa)$ is smooth or has only ordinary points, so
we cannot apply Deligne's results and assure that $H(\fa)$ is
finite. The finiteness of $H(\fa)$ is assured under some conditions
as in the case studied in \cite{CL}. As an application of
Theorem~\ref{Theorem:2irreducibledim1} we may extend Theorem~0 in
\cite{CL} as follows:

\begin{theorem}
\label{Theorem:firstintegral1} Let $\fa$ be a non-dicritic germ of
codimension one holomorphic  foliation with a singularity at the
origin $ 0 \in \mathbb C^n, n \geq 3$. Suppose that:

\begin{enumerate}

\item The codimension one component of the
tangent cone $ C(\fa)$ of $\fa$ is irreducible.

\item There is a basic generator of the projective holonomy with
multiplier  $1$ or a root of the unity of order $p^s$ for some prime
$p$ and $s\in\mathbb{N}^*$.

\end{enumerate}

Then  $\fa$ admits a holomorphic first integral.
\end{theorem}

We observe that if $ C(\fa)$ is irreducible of codimension one and
degree $p^s$ for some prime $p$ and $s \in \mathbb N^*$ then
condition (2) above is automatically satisfied (\cite{CL}  pages
219-220).

Given a codimension one foliation germ $\fa$ at $0 \in \mathbb C^n,
n \geq 3$ we consider maps $\tau \colon (\mathbb C^2,0) \to (\mathbb
C^n,0)$ {\it in general position} with respect to $\fa$ in the sense
of \cite{Ma-Mo} (see also \cite{C-LN-S} and \cite{Cerveau-Mattei}).
We shall call the inverse image $\tau^*(\fa)$ of $\fa$ by such a map
$\tau$, a {\it generic plane section} of $\fa$. In \cite{CL} it is
observed that for a nondicritical foliation germ, the irreducibility
of the tangent cone implies that the generic plane sections admit a
reduction by one single blow-up. In this sense we can extend
Theorem~\ref{Theorem:firstintegral1} above as follows:

\begin{theorem}
\label{Theorem:firstintegralreducible} Let $\fa$ be a non-dicritic
germ of codimension one holomorphic  foliation with a singularity at
the origin $ 0 \in \mathbb C^n, n \geq 3$. Suppose that:

\begin{enumerate}

\item The codimension one component of the
tangent cone $ C(\fa)$ of $\fa$ is a normal crossings divisor.

\item Each codimension one irreducible component $\Lambda_j$ of $ C(\fa)$
has degree $p_j^{s_j}$ for a prime $p_j$ and $s_j \in \mathbb N\cup
\{0\}$.

\item A generic plane section $\tau^*(\fa)$ of $\fa$, can be reduced
with a single blow-up $\pi\colon \widetilde {\mathbb C^2} \to
\mathbb C^2$.

\end{enumerate}

Then  $\fa$ admits a holomorphic first integral.
\end{theorem}

\section{Irreducible groups with an element tangent to the identity}
In this section we consider irreducible groups having a generator
with multiplier $1$, ie., of the form $f_k(z) = z + a_{k+1} z^{k+1}
+ \cdots$. Since any two generators are conjugate by hypothesis,
this implies that all elements in $G$ are tangent to the identity,
ie., $G$ is {\it flat}. We shall then start to prove
Theorem~\ref{Theorem:2irreducibledim1}.
\begin{proof}[Proof  of Theorem~\ref{Theorem:2irreducibledim1}: flat case]
There exists  a finite set of basic generators
$f_1,f_2,\ldots,f_{\nu +1}\in G$ such that:
\begin{enumerate}
\item[(a)] $f_1\circ f_2\circ\ldots\circ f_{\nu+1}=\id$.
\item[(b)] $f_i$ and $f_j$ are conjugate in $G$ for all $i,j$.
    \end{enumerate}
From (b) we have that all generators have the same multiplier.
Indeed, for $i\neq j$ there exists $g\in G $ of the form
$$g(z)=\alpha z+h.o.t.,\;\;\alpha\neq0$$
 such that $$f_i\circ g=g\circ f_j$$ hence we obtain
$$f_i^\prime(0)\cdot\alpha=f_j^\prime(0)\cdot\alpha$$ then
$$f_i^\prime(0)=f_j^\prime(0),\;\mbox{for all }i,j.$$
Since some generator is tangent to identity we have
$f^\prime_1(0)=\ldots=f^\prime_{\nu+1}(0)=1$. Then we have
$$f_i(z)=z+h.o.t.,\mbox{ for all }1\leq i\leq {\nu+1}.$$
 Consequently, every element $g\in G$ is of the form
$$g(z)=z+h.o.t.$$

In this case, we will show that
\begin{equation}\label{identity1}
f_1=f_2=\ldots=f_{\nu+1}=\id.
\end{equation}
Consequently $G$ is trivial.
The idea is to use Taylor series expansion and create an algorithm using derivation.
 \vglue.2in
Indeed, for $i\neq j$ by (b) there is $g\in G$ such that
\begin{equation}\label{dercomp01}
f_i\circ g=g\circ f_j.
\end{equation}

Derivating (\ref{dercomp01}) we have
\begin{equation}\label{dercomp11}
f_i^\prime(g)\cdot g^\prime=g^\prime(f_j)\cdot f_j^\prime.
\end{equation}
Also, derivating (\ref{dercomp11}) have
\begin{equation}\label{dercomp21}
f_i^{\prime\prime}(g)\cdot (g^\prime)^2+f_i^\prime(g)\cdot
g^{\prime\prime}=g^{\prime\prime}(f_j)\cdot
(f_j^\prime)^2+g^\prime(f_j)\cdot f_j^{\prime\prime}
\end{equation}
as $g(0)=f_i(0)=f_j(0)=0$ and
$g^\prime(0)=f^\prime_i(0)=f_j^\prime(0)=1$  we have in
(\ref{dercomp21})
$$f_i^{\prime\prime}(0)+g^{\prime\prime}(0)=g^{\prime\prime}(0)+f_j^{\prime\prime}(0)$$
therefore $$f^{\prime\prime}_i(0)=f_j^{\prime\prime}(0).$$
So there is $a_2=\dfrac{f^{\prime\prime}_i(0)}{2!}\in\co$ for all $i$ such that
$$f_i(z)=z+a_2z^2+h.o.t.,\mbox{ for all }1\leq i\leq \nu+1.$$
Then by (a) we have that $a_2=0$ therefore
$$f^{\prime\prime}_i(0)=f_j^{\prime\prime}(0)=0.$$

With the objective of obtaining a similar result for higher order
derivatives  we have to verify the following statement, variant of a
well-known result of Leibniz:

\begin{Claim}\label{Leibnitzgeral1}
For any  $\varphi$, $\psi$ diffeomorphisms and $n\in\na$, $n\geq3$ we have that
\begin{equation}\label{hipind1}
 (\varphi\circ\psi)^{(n)}=\varphi^{(n)}(\psi)\cdot (\psi^\prime)^n+\displaystyle\sum^{n-1}_{k=2}\varphi^{(k)}(\psi)\cdot R_k(\psi^\prime,\psi^{\prime\prime},\ldots,\psi^{(n+1-k)})+\varphi^\prime(\psi)\cdot \psi^{(n)}
 \end{equation}
where $R_k\in \mathbb C[x_1,...,x_{n+1-k}]$ are homogeneous polynomials of degree $k$ in $\co^{n+1-k}$ which has null coefficient in the monomial $x_1^k$.
\end{Claim}
\begin{proof}[Proof of Claim~\ref{Leibnitzgeral1}]
In fact, let us show this by induction on $n$, for $n=3$ we have
$$(\varphi\circ\psi)^{(3)}=\varphi^{\prime\prime\prime}(\psi)\cdot (\psi^\prime)^3+3\varphi^{\prime\prime}(\psi)\cdot \psi^\prime \cdot \psi^{\prime\prime}+\varphi^\prime(\psi)\cdot \psi^{\prime\prime\prime}$$
therefore equality is valid. Let us assume that equality is
satisfied  for $n$ by showing that it is valid for $n+1$. By the
hypothesis of induction (\ref{hipind1}) is valid. Derivating
(\ref{hipind1}) we have
\begin{equation}\label{derhipind1}
\begin{array}{c}
(\varphi\circ\psi)^{(n+1)}=\varphi^{(n+1)}(\psi)\cdot
(\psi^\prime)^{n+1}+n\varphi^{(n)}(\psi)\cdot (\psi^\prime)^{n-1}
\cdot \psi^{\prime\prime}+\\
\\\displaystyle\sum^{n-1}_{k=2}\left[\varphi^{(k+1)}(\psi)\cdot\psi^\prime\cdot
R_k(\psi^\prime,\psi^{\prime\prime},\ldots,\psi^{(n+1-k)})+\varphi^{(k)}(\psi)\cdot
\left(R_k(\psi^\prime,\psi^{\prime\prime},\ldots,\psi^{(n+1-k)})\right)^\prime
\right]\\ \\+\varphi^{\prime\prime}(\psi)\cdot \psi^\prime\cdot
\psi^{(n)}+\varphi^\prime(\psi)\cdot \psi^{(n+1)}.
\end{array}
\end{equation}
It is not difficult to see that
$\left(R_k(\psi^\prime,\psi^{\prime\prime},
\ldots,\psi^{(n+1-k)})\right)^\prime$ is a homogeneous polynomial
degree  $k$ in $\co^{n+2-k}$ which has null coefficient in the
monomial $x_1^k$. Denoting
$$\tilde{R}_n(\psi^\prime,\psi^{\prime\prime})=n(\psi^\prime)^{n-1}\cdot
\psi^{\prime\prime}+\psi^\prime\cdot R_{n-1}(\psi^\prime,\psi^{\prime\prime})$$

$$\tilde{R}_l(\psi^\prime,\psi^{\prime\prime},\ldots,\psi^{(n+2-l)})=\psi^\prime\cdot R_{l-1}(\psi^\prime,\psi^{\prime\prime},\ldots,\psi^{(n+2-l)})+\left(R_{l}(\psi^\prime,\psi^{\prime\prime},\ldots,\psi^{(n+1-l)})\right)^\prime,\;\mbox{ for }2<l<n.$$
and
$$ \tilde{R}_2(\psi^\prime,\psi^{\prime\prime},\ldots,\psi^{(n)})=\psi^\prime\cdot\psi^{(n)}+\left(R_{2}(\psi^\prime,\psi^{\prime\prime},\ldots,\psi^{(n-1)})\right)^\prime$$

we can re-write (\ref{derhipind1}) as
$$ (\varphi\circ\psi)^{(n+1)}=\varphi^{(n+1)}(\psi)\cdot
(\psi^\prime)^{n+1}+\displaystyle\sum^{n}_{l=2}\varphi^{(l)}(\psi)
\cdot
\tilde{R}_l(\psi^\prime,\psi^{\prime\prime},\ldots,\psi^{(n+2-l)})+
\varphi^\prime(\psi)\cdot \psi^{(n+1)}
 $$
where $\tilde{R}_l$ are homogeneous polynomials of degree $l$ in
$\co^{n + 2-l}$  which has null coefficient in the monomial $x_1^l$.
This proves Claim~\ref{Leibnitzgeral1}.
\end{proof}

Our next step is:
\begin{Claim}\label{null1}
\begin{equation}\label{iterednull1}
f_i^{(n)}(0)=f_j^{(n)}(0)=0
\end{equation}
 for all $n\in\mathbb{N}$, $n\geq3$.
 \end{Claim}
 \begin{proof}[Proof of Claim~\ref{null1}]
 As before, we shall use induction on $n$ to prove the claim.
 For $n=3$, derivating (\ref{dercomp21}) we have to
\begin{equation}\label{dercomp3p1}
f_i^{\prime\prime\prime}(g)\cdot (g^\prime)^3+3f_i^{\prime\prime}(g)
\cdot g^\prime\cdot g^{\prime\prime}+f_i^\prime(g)\cdot
g^{\prime\prime\prime}=g^{\prime\prime\prime}(f_j)\cdot
(f_j^\prime)^3+ 3g^{\prime\prime}(f_j)\cdot f^\prime_j\cdot
f^{\prime\prime}_j+g^\prime(f_j)\cdot f^{\prime\prime\prime}_j
\end{equation}
as $g(0)=f_i(0)=f_j(0)=0$,
$g^\prime(0)=f^\prime_i(0)=f_j^\prime(0)=1$  and
$f^{\prime\prime}_i(0)=f_j^{\prime\prime}(0)=0$ we have in
(\ref{dercomp3p1})
$$f_i^{\prime\prime\prime}(0)+g^{\prime\prime\prime}(0)=g^{\prime\prime\prime}(0)+f_j^{\prime\prime\prime}(0)$$therefore $$f^{\prime\prime\prime}_i(0)=f_j^{\prime\prime\prime}(0)$$
hence there is $a_3=\dfrac{f^{\prime\prime\prime}_i(0)}{3!}\in\co$
for all $i$  such that
$$f_i(z)=z+a_3z^3+h.o.t.,\mbox{ for all }1\leq i\leq \nu+1.$$
Then by (a) we have that $a_3=0$ therefore
$$f^{\prime\prime\prime}_i(0)=f_j^{\prime\prime\prime}(0)=0.$$
Suppose the statement is satisfied for $ 3\leq l <n$ by showing that
it is valid for $n$. Now derivating $n$-times (\ref{dercomp01}) and
using Claim~\ref{Leibnitzgeral1} we have to
\begin{equation}\label{dercompnp1}
\begin{array}{c}
f_i^{(n)}(g)\cdot (g^\prime)^n+\displaystyle\sum^{n-1}_{l=2}f_i^{(l)}(g)\cdot R_l(g^\prime,g^{\prime\prime},\ldots,g^{(n+1-l)})+f_i^\prime(g)\cdot g^{(n)}\\
\\
=g^{(n)}(f_j)\cdot (f_j^\prime)^n+\displaystyle\sum^{n-1}_{l=2}g^{(l)}(f_j)\cdot \tilde{R}_l(f_j^\prime,f_j^{\prime\prime},\ldots,f_j^{(n+1-l)})+g^\prime(f_j)\cdot f_j^{(n)}
\end{array}
\end{equation}
where $R_l$ and $\tilde{R}_l$ are homogeneous polynomials  of degree
$l$ in $\co^{n+1-l}$ which has null coefficient in the monomial
$x_1^l$. By the induction hypothesis (\ref{iterednull1}) it is valid
for all $3\leq l <n$ so
 $$f^{(l)}_i(0)=f_j^{(l)}(0)=0,\;\mbox{for all }3\leq l<n,$$
 $g(0)=f_i(0)=f_j(0)=0$, $g^\prime(0)=f^\prime_i(0)=f_j^\prime(0)=1$
 and $f^{\prime\prime}_i(0)=f_j^{\prime\prime}(0)=0$ we have in (\ref{dercompnp1})
 $$f_i^{(n)}(0)+g^{(n)}(0)=g^{(n)}(0)+f_j^{(n)}(0)$$
 therefore $$f^{(n)}_i(0)=f_j^{(n)}(0)$$hence there is
 $a_n=\dfrac{f^{(n)}_i(0)}{n!}\in\co$ for all $i$ such that
$$f_i(z)=z+a_nz^n+\ldots,\mbox{ for all }1\leq i\leq \nu+1.$$
So by (a) we have that $a_n=0$ therefore
$$f^{(n)}_i(0)=f_j^{(n)}(0)=0.$$ This ends the proof of Claim~\ref{null1}.
\end{proof}
By Claim~\ref{null1} we have equality (\ref{identity1}) and
therefore $G=\{id\}$.  This proves
Theorem~\ref{Theorem:2irreducibledim1} in  case some generator is
tangent to the identity.
\end{proof}

\section{Irreducible groups with a primitive root of the unity}
Recall that if $f_1,...,f_{\nu+1}$ is a set of generators of an
irreducible group $G$ as in
definition~\ref{Definition:irreduciblegroup} then we have
$f_k^\prime(0)=\mu$ for all $k$, with $\mu ^{\nu+1}=1$. Thus the
multiplier of any generator is a root of the unity. Its order is a
divisor of $\nu+1$ but it is not necessarily of the form $p^s$ for a
prime $p$ (see examples~\ref{Contraexample1} and
\ref{Example:more}). In Theorem~\ref{Theorem:2irreducibledim1} we
assume that $\mu$ has order $p^s$. The proof of this theorem goes
partially as the one given in \cite{CL} for
Theorem~\ref{Theorem:CL}. Nevertheless, we shall make use of our
preceding case in Theorem~\ref{Theorem:2irreducibledim1} and of some
techniques borrowed from \cite{CL}.

\begin{proof}[Proof  of Theorem~\ref{Theorem:2irreducibledim1}: remaining case]

Suppose that the multiplier $\mu\neq 1$ of some (and therefore of
each)  basic generator is a  root of the unit of order $p^s$ for
some prime $p$ and some $s \in \mathbb N^*$. By the above remark
$p^s$ divides $\nu+1$. We will perform an adaptation of the
arguments in the proof of \cite[Theorem 1]{CL} in order to conclude
that the group is finite. The idea is to show by a formal algorithm
that $G$ is formally linearizable.  More precisely, let us prove
that:

\begin{Claim} If the group $G$ is linear (by a conjugation) to order $k$, that is,
$$f_i(z)=\mu z+t_iz^{k+1}+h.o.t,\;\;\mbox{where }t_i\in\co$$ for all $i$
then either the $t_i=0$ for all $i$ (and the group is  linearized to
order $k+1$) or $\nu+1$ does not divide $k$ and
$t_1=t_2=\ldots=t_{\nu+1}$.
\end{Claim}

In the latter case, we can linearize $G$ to order $k+1$ by
conjugating by
$$z\mapsto z+\dfrac{t_1}{\mu-\mu^{k+1}}z^{k+1}.$$
This produces a convergent algorithm in the Krull  topology
producing a complex conjugate diffeomorphism $G$ to the group
generated by the rational rotation $R(z)=\mu z$.
\begin{proof}[proof of the claim]
Suppose that $\nu+1$ divides $k$, we define the  following
application

$$\begin{array}{c c c c} \varphi:&G&\to& \co \\
&az+bz^{k+1}+h.o.t.&\mapsto&\dfrac{b}{a}  \end{array}$$

which defines a morphism from  $(G,\circ)$ into $(\co,+)$: indeed
given $f,g\in G$ we write
$$f(z)=az+bz^{k+1}+h.o.t.\;\;\mbox{and}\;\;g(z)=cz+dz^{k+1}+h.o.t.$$
and how $f\circ g\in G$ is written
$$f\circ g(z)=acz+(ad+bc^{k+1})z^{k+1}+h.o.t.$$
so we have
$$\varphi(f\circ g)=\dfrac{ad+bc^{k+1}}{ac}.$$
Since $\nu+1$ divides $k$ and $c=\mu^n$ for some integer $n$  then
we have $c^{k+1}=c$ therefore
$$\varphi(f\circ g)=\dfrac{ad+bc}{ac}=\dfrac{b}{a}+\dfrac{d}{c}=\varphi(f)+\varphi(g).$$
Now from (a) we have
$$\begin{array}{c} \varphi(f_1\circ f_2\circ\ldots\circ f_{\nu+1})=\varphi(\id)=0\\ \\
\varphi(f_1)+\varphi(f_2)+\ldots+\varphi(f_{\nu+1})=0\\ \\
\dfrac{t_1}{\mu}+\dfrac{t_2}{\mu}+\ldots+\dfrac{t_{\nu+1}}{\mu}=0.
  \end{array}$$
 Now for $i\neq j$ by (b) we have $h\in G$ such that
$$f_i\circ h=h\circ f_j$$
applying the morphism we have

$$\begin{array}{c} \varphi(f_i\circ h)=\varphi(h\circ f_j)\\ \\
\varphi(f_i)+\varphi(h)=\varphi(h)+\varphi(f_j)\\ \\
\varphi(f_i)=\varphi(f_j)\\
\\
\dfrac{t_i}{\mu}=\dfrac{t_j}{\mu}
\\
\\
t_i=t_j.
  \end{array}$$
  Therefore $t_1=t_2=\ldots=t_{\nu+1}=0$.

Suppose that $\nu+1$ does  not divide $k$, we define the following
application $\psi\colon G\to \Aff(\co), \,
\psi(az+bz^{k+1}+h.o.t.)=\dfrac{az+b}{a^{k+1}}$, which defines a
morphism from  $(G,\circ)$ into $(\Aff(\co),\circ)$: indeed given
$f,g\in G$ we write
$$f(z)=az+bz^{k+1}+h.o.t.\;\;\mbox{and}\;\;g(z)=cz+dz^{k+1}+h.o.t.$$
and since $f\circ g\in G$ is written as
$$f\circ g(z)=acz+(ad+bc^{k+1})z^{k+1}+h.o.t.$$
so we have
$$\psi(f\circ g)=\dfrac{acz+(ad+bc^{k+1})}{(ac)^{k+1}}$$
on the other hand
$$\psi(f)\circ\psi(g)=\dfrac{a\left(\dfrac{cz+d}{c^{k+1}}\right)+b}{a^{k+1}}=
\dfrac{acz+(a+bc^{k+1})}{(ac)^{k+1}}=\psi(f\circ g).$$

Denote by $G_0$ the image of $G$ by $\psi$. Therefore $G_0$  is an
affine group generated by the transformations
$g_1,g_2,\ldots,g_{\nu+1}$, pairwise  conjugate in $G_0$, where
$$g_i(z)=\dfrac{\mu z+t_i}{\mu^{k+1}}.$$

We now apply  the following lemma whose proof is found in \cite{CL}
(page 222):

\begin{Lemma} Let $\eta$ be  a $l$-th root of the unit, $l>1$,
 $\beta_1,\beta_2,\ldots,\beta_{r+1}\in\co$ and $\Gamma$
 an affine group generated by the transformations
  $h_i(z)=\eta z+\beta_i$, $i=1,2,\ldots,r+1$.
Then the $h_i$'s are pairwise conjugate in $\Gamma$ if and  only if
either $l$ has two distinct prime divisors or $l=q^m$, for some
prime $q$ and some  $m\in\mathbb{N}^*$ and
$\beta_1=\beta_2=\ldots=\beta_{r+1}$.
\end{Lemma}

Taking $\eta=\dfrac{1}{\mu^k}$, $l=p^s$, $r=\nu$  and
$\beta_i=\dfrac{t_i}{\mu^{k+1}}$ ($i=1,...,\nu+1$) by the previous
lemma we have $p^s=q^m$, $q$ prime, $m\in\mathbb{N}^*$ and
$$\dfrac{t_1}{\mu^{k+1}}=\dfrac{t_2}{\mu^{k+1}}=\ldots=\dfrac{t_{\nu+1}}{\mu^{k+1}}.$$
Therefore $q=p$, $m=s$ and
$$t_1=t_2=\ldots=t_{\nu+1}.$$
This proves the claim. \end{proof}

The proof of Theorem~\ref{Theorem:2irreducibledim1} is now complete.
\end{proof}

As for the case of groups of germs of real analytic  diffeomorphisms
we promptly obtain:

\begin{Corollary}
Let $f_1,f_2,\ldots,f_{\nu +1}\in \Diff^w(\rr,0)$ be germs of real
analytic diffeomorphisms and let $G\subset \Diff^w(\rr,0)$ the group
generated by them. Assume that:
\begin{enumerate}
\item[(a)] $f_1\circ f_2\circ\ldots\circ f_{\nu+1}=\id$.
\item[(b)] $f_i$ and $f_j$ are conjugate in $G$ for all $i,j$.
\item[(c)] Some $f_j$ has a multiplier of order $p^s$ for
some prime $p$ and $s \in \mathbb N$.
    \end{enumerate}
    Then $G$ is finite of order $\leq 2$. If $s=0$ then $G$ is trivial.
\end{Corollary}

\begin{proof}
It is enough to consider the group $G^\mathbb C\subset \Diff(\mathbb
C,0)$ generated by the complexification of the maps
$f_j,\;j=1,...,\nu+1$. The group satisfies the hypotheses of
Theorem~\ref{Theorem:2irreducibledim1} and therefore it is finite.
Since the complexification defines an injective morphism $G \to
G^{\mathbb C}$ the group $G$ must finite. Since $G$ is analytically
linearizable, the conclusion follows from the fact that a finite
order  germ of real analytic diffeomorphism has order one or two.
\end{proof}

\section{Examples and counterexamples}
\label{section:examples} Now we give some examples illustrating our
results and the necessity of the hypotheses we make. The first two
examples show that the condition on the order of the multiplier of a
generator cannot be dropped in
Theorem~\ref{Theorem:2irreducibledim1}.

\begin{Example}\label{Contraexample1}
{\rm Consider $G\subset \Diff(\mathbb C,0)$ the subgroup generated
by the maps
$$f_1(z)=f_2(z)=f_3(z)=f_4(z)=\dfrac{z}{a}, f_5(z)=\dfrac{z}{a+z}\;\;
\mbox{ and }\;\;f_6(z)=\dfrac{z}{a-a^5z}$$where $a^6=1$ so that
$a=\dfrac{1}{1-a}$. We claim  that $G$ is irreducible and not
abelian. The first condition is satisfied
$$\begin{array}{r c l}f_1\circ f_2\circ f_3\circ
f_4\circ f_5\circ f_6(z)&=&f_1\circ f_2\circ f_3\circ
f_4\left(\dfrac{z}{a^2}\right)=\dfrac{z}{a^6}=z.
 \end{array}$$
To check the second condition take  $g_1(z)=f_1\circ f_5\circ
f_1^4(z)=\dfrac{z}{1+az}$ then  $g_1\in G$ and note that

$$f_5\circ g_1(z)=\dfrac{z}{a+(1+a^2)z}$$
$$g_1\circ f_1(z)=\dfrac{z}{a+az}.$$ Since
$a=\dfrac{1}{1-a}$ we have $1+a^2=a$ therefore $f_5\circ
g_1=g_1\circ f_1$. Similarly, since  $a^3=-1$ and $1+a^2=a$ we have
$f_1\circ g_2=g_2\circ f_6$. Finally, because  $1+a-a^5=1+a+a^2$ we
have  $f_5\circ g_3=g_3\circ f_6$. Consequently the $f_j$ are
pairwise conjugate in the group $G$. Finally, we observe that $G$ is
not abelian and in particular, it is not analytically linearizable:
indeed, if $g$ linearizes $G$, then $g\circ f_1\circ g^{-1}= f_1$. }
\end{Example}

Other examples of irreducible groups which are not finite are
briefly listed below:

\begin{Example}
\label{Example:more}{\rm  (1) Let $G_{10}\subset \Diff(\mathbb C,0)$
be the subgroup finitely generated by the maps
$$f_1(z)=\cdots=f_8(z)=\dfrac{z}{a}, f_9(z)=\dfrac{z}{a+z}\;\;
\mbox{ and }\;\;f_{10}(z)=\dfrac{z}{a-a^9z}$$where $a^{10}=1$  so
that $a^3+a=\dfrac{1}{1-a}$.

(2)  Consider $G_{12}\subset \Diff(\mathbb C,0)$ the subgroup
finitely generated by the maps
$$f_1(z)=\cdots=f_{10}(z)=\dfrac{z}{a},
f_{11}(z)=\dfrac{z}{a+z}\;\;\mbox{ and }\;\;
f_{12}(z)=\dfrac{z}{a-a^{11}z}$$where $a^{12}=1$ so that
$a=\dfrac{1}{1-a}$.

(3) Take $G_{12}^\prime \subset \Diff(\mathbb C,0)$ as the subgroup
finitely generated by the maps
$$f_1(z)=\cdots=f_{10}(z)=\dfrac{z}{a},
f_{11}(z)=\dfrac{z}{a+z}\;\;\mbox{ and }\;\;
f_{12}(z)=\dfrac{z}{a-a^{11}z}$$where $a^{12}=1$ so that
$a^3+a^2=\dfrac{1}{1-a}$.

(4) Let $G_{14}\subset \Diff(\mathbb C,0)$ be the subgroup finitely
generated by the maps
$$f_1(z)=\cdots=f_{12}(z)=\dfrac{z}{a},
f_{13}(z)=\dfrac{z}{a+z}\;\;\mbox{ and }\;\;
f_{14}(z)=\dfrac{z}{a-a^{13}z}$$where $a^{14}=1$ so that
$a^5+a^3+a=\dfrac{1}{1-a}$.

(5) Consider $G_{18}\subset \Diff(\mathbb C,0)$ the subgroup
finitely generated by the applications
$$f_1(z)=\cdots=f_{16}(z)=\dfrac{z}{a},
f_{17}(z)=\dfrac{z}{a+z}\;\;\mbox{ and }\;\;
f_{18}(z)=\dfrac{z}{a-a^{17}z}$$where $a^{18}=1$ so that
$a=\dfrac{1}{1-a}$.

(6) Put $G_{18}^\prime \subset \Diff(\mathbb C,0)$ as the subgroup
finitely generated by the applications
$$f_1(z)=\cdots=f_{16}(z)=\dfrac{z}{a},
f_{17}(z)=\dfrac{z}{a+z}\;\;\mbox{ and }\;\;
f_{18}(z)=\dfrac{z}{a-a^{17}z}$$where $a^{18}=1$ so that
$a^5+a^4+a^3=\dfrac{1}{1-a}$.
}
\end{Example}

We now show that condition (b) in the definition of irreducibility
(Definition~\ref{Definition:irreduciblegroup}) is essential in
Theorem~\ref{Theorem:2irreducibledim1}.
\begin{Example}
\label{Example:essential} {\rm Indeed, given $p \in \mathbb N$
consider $G=<f,f^{-1}>\subset\Diff(\co,0)$ where
$f(z)=\frac{z}{(1-z^p)^{1/p}}$, $f^{-1}(z)=\frac{z}{(1+z^p)^{1/p}}$.
Take $h(z)=e^{\pi i/p}z$. Note that $f\circ f^{-1}=\id$ and
$f(h(z))=\frac{e^{\pi i/p}z}{(1+z^p)^{1/p}}= e^{\pi
i/p}\frac{z}{(1+z^p)^{1/p}}=h(f^{-1}(z))$. Thus $f$ and $f^{-1}$ are
conjugate in $\Diff(\mathbb C,0)$,  but clearly not in $G$. Note
also that $f^{(n)}(z)=\frac{z}{(1-nz^p)^{1/p}}\neq z$ for all
$n\in\mathbb{N}$. Therefore $G$ it is not finite. }
\end{Example}


\section{Germs of foliations: proof of
Theorems~\ref{Theorem:firstintegral1} and
~\ref{Theorem:firstintegralreducible}}

We consider a germ $\fa$ of a codimension one holomorphic  foliation
singular at the origin $0 \in \mathbb C^n, n \geq 3$. We consider
the blow-up $\pi \colon \widetilde {\mathbb C^n} \to \mathbb C^n$ of
$\mathbb C^n$ at the origin $0 \in \mathbb C^n$ and denote by
$\tilde \fa$ the lifted foliation $\tilde \fa= \pi^*(\fa)$. We also
denote by $E=\pi^{-1}(0)\cong \mathbb P^{n-1}$ the exceptional
divisor of the blow-up.

\begin{proof}[Proof of Theorem~\ref{Theorem:firstintegral1}]
The proof goes as the proof of Theorem~0 in \cite{CL}. The key point
is then the following fact:
\begin{Claim}
The projective holonomy $H(\fa)$ is a finite group.
\end{Claim}
This is proved as an immediate consequence of the irreducibility of
the tangent cone, the fact that the projective holonomy is generated
by the basic holonomy generators, and
Theorem~\ref{Theorem:2irreducibledim1}.
\end{proof}

\begin{proof}[Proof of Theorem~\ref{Theorem:firstintegralreducible}]
Again the proof has steps in common with the proof of Theorem~0 in
\cite{CL}. As already mentioned (see the paragraph before
Theorem~\ref{Theorem:firstintegralreducible}), it is observed in
\cite{CL} that the irreducibility of the tangent cone implies that a
generic plane section of $\fa$ can be reduced with a single blow-up.
We have this same property by hypothesis now. The main point is
then, again, the following fact:
\begin{Claim}
The projective holonomy $H(\fa)$ is a finite group.
\end{Claim}
\begin{proof}[proof of the claim]
Let $L(\fa)=E\setminus  C(\fa)$ be the leaf of $\tilde \fa$ in the
blow-up $\widetilde{\mathbb C^n}$.  We shall prove that the image
$H(\fa)$ of any holonomy representation $\Hol\colon \pi_1(L(\fa))\to
\Diff(\mathbb C,0)$ is a finite group. Notice that the codimension
$\geq 3$ singularities of $\tilde \fa$ do not affect the fundamental
group of $L(\fa)$. Therefore  $L(\fa)$ has the homotopy type of the
complement of a normal crossing divisor in $\mathbb P^n$. By
Deligne's theorem this fundamental group is abelian. Moreover, this
group is generated by simple  loops around the codimension one
irreducible components of $ C(\fa)$. This proves that the group
$\pi_1(L(\fa))$ is the direct product of the groups
$\pi_1(E\setminus \Lambda_j), j=1,...,r$ where $\Lambda_j\subset E,
j=1,...,r$ are the codimension one irreducible components of
$C(\fa)\subset E$. This implies that any holonomy representation
$\Hol\colon \pi_1(L(\fa)) \to \Diff(\mathbb C,0)$ has as image an
abelian group $H(\fa)\subset \Diff(\mathbb C,0)$ which is a direct
product of groups $H_j=\Hol(\pi_1(E\setminus \Lambda_j))\subset
\Diff(\mathbb C,0)$. Each group $H_j\subset \Diff(\mathbb C,0)$ is
irreducible and is finite because of the hypothesis on the degree of
the components $\Lambda_j$ and of the conclusion of
Theorem~\ref{Theorem:2irreducibledim1}. Thus the group $H(\fa)$ is
abelian as well.
\end{proof}
Once we have proved that the projective holonomy is finite and
knowing that the generic hyperplane sections of the foliation can be
reduced with a single blow-up we can proceed as in \cite{CL} and
conclude.

\end{proof}

\section{Foliations in  projective spaces and local examples}

 We turn our attention to
the case of foliations in the  projective space $\mathbb P^n$ of
dimension $n \ge 2$. It is well-known that a codimension one
holomorphic foliation (with singularities) on a complex projective
space $\mathbb P^n$ can be defined in homogeneous coordinates
$(z_1,...,z_{n+1})$ in $\co^{n+1}$ by a differential 1-form
$$
\om = \sum_{i=1}^{n+1} a_i(z)dz_i \qquad z \in \co^{n+1} $$ where
the $a_i(z)$ are homogeneous polynomials of the same degree (without
common factors and) satisfying $\sum\limits_{i=1}^{n+1} z_i\,a_i(z)
= 0$,
 and the integrability
condition $\om \wedge d\om = 0$. The singular set of $\om$ is
$S(\om) := \{z \in \co^{n+1}; a_j(z) = 0\,,\,\, j=1,\ldots,n+1\}$.
We denote by $\widehat\fa(\om)$  the codimension one singular
holomorphic foliation on $\co^{n+1}$ defined by $\om$ with singular
set $\sing (\om)$. If $\pi\colon \co^{n+1}\backslash\{0\} \to
\mathbb P^n$ denotes the canonical projection, this induces a
foliation $\fa = \fa(\om)$ on $\mathbb P^n$ with singular set
$\sing\,\fa = \pi(S(\om))$ and $\pi^*(\fa) = \widehat\fa(\om)$. As
we have remarked above there is no loss of generality if we assume
that $codim\,\sing\,\fa \ge 2$. Before going into our main example,
we must recall the notion of Kupka singularity:

 \begin{Definition}[Kupka singularity] {\rm

Let $\omega=\sum_{j=1}^na_jdx_j$ be a holomorphic integrable 1-form
on a neighborhood $\U$ of $0\in \co^n$. We assume that $n\ge3$ and
that $\sing(\omega)=\{p\in \U\mid \omega(p)=0\}$ is a codimension
$\geq 2$ analytic subset. A singularity $q\in\sing(\omega)$ is a
{\it Kupka singularity\/} if $d\omega(q)\ne0$. We denote by
$K(\omega)$ the subset of Kupka singularities.}
 \end{Definition}
The main properties of the Kupka singularities are stated in
\cite{Omegar}, \cite{Kupka}. These singularities are of local
product type, by a foliation in dimension two.

\begin{Example}
\label{Example:Kupka}{\rm Let $k \in \mathbb N, k\geq 2$ be given.
Let also $a,\;b,\;c,\;\alpha,\;\beta,\;\gamma\in\co^*$ be fixed.  We
consider the integrable homogeneous 1-form $\Omega$  in $\co^4$ with
coordinates $(x,y,z,w)$  by:
$$\begin{array}{l l
l}\Omega&=&\left[-x^2(ay^{k-1}+bz^{k-1}+cw^{k-1})+\alpha
y^{k+1}+\beta z^{k+1}+\gamma
w^{k+1}\right]dx\\&&\\&&+xy^{k-2}(ax^2-\alpha
y^2)dy+xz^{k-2}(bx^2-\beta z^2)dz+xw^{n-2}(cx^2-\gamma
w^2)dw\end{array}$$

Notice that $\Omega ({\vec R})=0$ for the radial vector field ${\vec
R}=x\frac{\partial}{\partial x} + y \frac{\partial}{\partial y}+ z
\frac{\partial}{\partial z} + w\frac{\partial}{\partial w}$. Thus
$\Omega$ defines a codimension one holomorphic foliation $\fa$ of
degree $k$ in $\mathbb P^3$. This foliation leaves invariant  the
hyperplane of infinity $H=\pi(x=0)$, where
$\pi:\co^4\setminus\{0\}\to \mathbb P^3$ is the canonical
projection.
\begin{Claim}
We have $\sing(\fa)\cap H =S$ is irreducible and smooth. Moreover,
all singularities of $\fa$ in $H$ are of Kupka type.
\end{Claim}
\begin{proof}[proof of the claim]
First of all we note that $$H\cap \sing(\fa)=\{[0:y:z:w];\;\alpha
y^{k+1}+\beta z^{k+1} +\gamma w^{k+1}=0\}$$ is a smooth irreducible
hypersurface of degree $k+1$. Furthermore
$$d\left(\dfrac{\Omega}{F}\right)=0$$where $F$  is
the homogeneous polynomial of degree $k+2$ given by
$$F=x^2\left[x\left(a\dfrac{y^{k-1}}{k-1}+
b\dfrac{z^{k-1}}{k-1}+c\dfrac{w^{k-1}}{k-1}\right)
-\alpha\dfrac{y^{k+1}}{k+1}-\beta\dfrac{z^{k+1}}{k+1}
-\gamma\dfrac{w^{k+1}}{k+1}+x^{k+1}\right].$$We can write  $F=xQ$,
where $Q$ is the polynomial of degree $k$ given by
$$Q=x^2\left(a\dfrac{y^{k-1}}{k-1}+b\dfrac{z^{k-1}}{k-1}
+c\dfrac{w^{k-1}}{k-1}\right)-\alpha\dfrac{y^{k+1}}{k+1}
-\beta\dfrac{z^{k+1}}{k+1}-\gamma\dfrac{w^{k+1}}{k+1}+x^{k+1}.$$
Thus
$$\Omega=x^{k+2} d\left(\dfrac{Q}{x^{k+1}}\right)=-(k+1)Qdx+xdQ.$$
Let now $p=(x,y,z,w) \in \mathbb C^4$ be such that
$\Omega(p)=d\Omega (p)=0$. From $\Omega(p)=0$ we have $Q(p)=x(p)=0$.
From $\Omega=-(k+1)Qdx+xdQ$ we get $d\Omega=-(k+2) dQ \wedge dx$.
Thus $d\Omega(p)=0$ implies $dQ(p)\wedge dx(p)=0$. From the
expression of $Q$ above, taking into account that $x(p)=0$, we have
$dQ(p)\wedge dx(p)=0 \implies y(p)=z(p)=w(p)=0$. Since
$(x,y,z,w)(p)\ne 0$ in homogeneous coordinates, we have that every
singularity in $S$ is of Kupka type.

A final remark is that $\dfrac{Q}{x^{k+1}}$ is a first meromorphic
integral of $\fa$.
\end{proof}
}
\end{Example}

The above example allows us to construct local examples of the
situation we deal with in Theorem~\ref{Theorem:firstintegral1}.

\begin{Example}
{\rm   Example~\ref{Example:Kupka}, originally placed in the
projective space $\mathbb P^3$ may be used in order to produce a
local example, illustrating our
Theorem~\ref{Theorem:firstintegralreducible}. Indeed, the idea is to
find a germ of foliation at $0 \in \mathbb C^3$ that produces after
a blow-up at the origin, a foliation having an invariant exceptional
divisor which will be bimeromorphically mapped into the hyperplane
at infinity $H\subset \mathbb P^3$ of the above example.  For this
consider the one-form $\Omega$ given in $\co^3$ by
$$\Omega=\left[-x^2(y^2+z^2)+y^4+z^4\right]dx+xy(x^2-y^2)dy+xz(x^2-z^2)dz.$$
Using the relations
$$x=\dfrac{1}{X},\;y=\dfrac{Y}{X^2},\;z=\dfrac{Z}{X^2}$$
we have the pull-back one-form
$$\begin{array}{r c l} \tilde{\Omega}&=&\left[-\dfrac{1}{X^2}\left( \dfrac{Y^2}{X^4}+\dfrac{Z^2}{X^4}\right)+\dfrac{Y^4}{X^8}+\dfrac{Z^4}{X^8}\right]d\left(\dfrac{1}{X}\right) +\dfrac{1}{X}\dfrac{Y}{X^2}\left(\dfrac{1}{X^2}-\dfrac{Y^2}{X^4}\right)d\left(\dfrac{Y}{X^2}\right)
\\&&\\&&+\dfrac{1}{X}\dfrac{Z}{X^2}\left(\dfrac{1}{X^2}-\dfrac{Z^2}{X^4}\right)d\left(\dfrac{Z}{X^2}\right)\\&&\\
&=&
\left[\dfrac{X^2(Y^2+Z^2)-Y^4-Z^4}{X^{10}}\right]dX+\dfrac{Y(X^2-Y^2)}{X^{10}}\left(XdY-2YdX\right)+\dfrac{Z(X^2-Z^2)}{X^{10}}\left(XdZ-2ZdX\right)

 \end{array} $$
and then
 $$\begin{array}{r c l} X^{10}\tilde{\Omega}
 &=&
 \left[-X^2(Y^2+Z^2)+Y^4+Z^4\right]dX+XY(X^2-Y^2)dY+XZ(X^2-Z^2)dZ.\end{array}$$
Thus by construction
$\alpha(X,Y,Z)=\left[-X^2(Y^2+Z^2)+Y^4+Z^4\right]dX+XY(X^2-Y^2)dY+XZ(X^2-Z^2)dZ$
is an integrable one-form in $\mathbb C^3$, it admits the following
first integral
\[
f=\frac{1}{X^4} \big[X^2 \big(\frac{Y^2 + Z^2}{2}\big) -
\frac{Y^4}{4} - \frac{Z^4} {4}\big].
\]

We then consider the birational  change of coordinates
$(x,y,z)=(1/x,y,z)=(X,Y,Z)$ which gives
\[
f(x,y,z)=x^2(y^2+z^2)/2 -x^4y^4/4-x^4z^4/4
\]
and the one-form
\[
\omega=[(y^2 + z^2) + x^2 ( y^4 + z^4)]dx + xy(1- x^2 y^2)dy + zx (1
- x^2 z^2)dz.
\]
Then, finally, $\omega$ is an integrable one-form in a neighborhood
of the origin $0 \in \mathbb C^3$. The corresponding foliation
$\fa_\omega: \omega=0$ has a tangent cone given by the degree $3$
homogeneous equation $(P_3=0)\subset \mathbb P^2$ for
$P_3=x(y^2+z^2)+xy^2+xz^2=2x(y^2+z^2)$. Therefore $C(\fa_\omega)$ is
a normal crossings divisor with components given by three projective
lines, induced by  $x=0, y + iz=0, y - iz=0$. Because of the Kupka
type of the singularities at infinity of the starting foliation
$\fa_\Omega: \Omega=0$, we have that a generic plane section
$\tau^*(\fa_\omega)$ of $\fa_\omega$, can be reduced with a single
blow-up $\pi\colon \widetilde {\mathbb C^2} \to \mathbb C^2$,
verifying then condition (3) in
Theorem~\ref{Theorem:firstintegralreducible}.

}
\end{Example}

\subsection{Formal case}

Theorem~\ref{Theorem:2irreducibledim1} also holds in the formal
framework. Indeed, the proofs are  {\it ipsis litteris} the same.
Denote by $\widehat\Diff(\mathbb C,0)$ the group of formal complex
diffeomorphims fixing the origin $0 \in \mathbb C$. An element $\hat
f\in \widehat\Diff(\mathbb C,0)$ is a formal power series $\hat
f(z)=\mu z + \sum\limits_{j=1}^\infty a_j z^j\in \mathbb C[[z]]$
with $\mu \ne 0$.

We may then state:

\begin{Theorem}
\label{Theorem:2irreducibledim1formal} Let $\hat G\subset \widehat
\Diff(\mathbb C,0)$ be an irreducible  group, ie., a finitely
generated subgroup with generators $\hat f_1,...,\hat f_{\nu+1}$
pairwise conjugate in $\hat G$ and such that $\hat f_{\nu+1}\circ
\cdots \circ \hat f_1=\id$. If the multiplier of any basic generator
is  a $p^s$-root of the unit, where $p$ is prime and
$s\in\mathbb{N}$, then $\hat G$ is a finite group, trivial in the
first case $s=0$.
\end{Theorem}

\bibliographystyle{amsalpha}

\end{document}